\documentclass[12pt]{amsart}
\usepackage{amssymb,amstext,amsmath,amscd,amsthm,amsfonts,enumerate,graphicx,latexsym,color,stmaryrd}
\usepackage[all]{xy}
\usepackage{comment}

\usepackage{hyperref}
\newtheorem{thm}{Theorem}[section]
\newtheorem*{thm*}{Theorem}

\newtheorem{cor}[thm]{Corollary}
\newtheorem{lem}[thm]{Lemma}
\newtheorem{prop}[thm]{Proposition}
\theoremstyle{definition}
\newtheorem{conv}[thm]{Convention}
\newtheorem{dfn}[thm]{Definition}
\newtheorem*{dfn*}{Definition}
\newtheorem{rem}[thm]{Remark}

\newtheorem*{conj*}{Conjecture}
\newtheorem{ex}[thm]{Example}

\theoremstyle{remark}
\newtheorem*{ac}{Acknowledgments}

\newtheorem*{claim*}{Claim}

\numberwithin{equation}{thm}

\def\PP{\mathbb{P}}

\def\ZZ{\mathbb{Z}}

\def\cD{\mathcal{D}}
\def\cE{\mathcal{E}}
\def\cF{\mathcal{F}}
\def\cG{\mathcal{G}}
\def\cK{\mathcal{K}}
\def\cI{\mathcal{I}}

\def\cO{\mathcal{O}}
\def\cP{\mathcal{P}}
\def\cQ{\mathcal{Q}}
\def\cS{\mathcal{S}}
\def\cT{\mathcal{T}}

\def\cX{\mathcal{X}}
\def\cY{\mathcal{Y}}

\def\fm{\mathfrak{m}}
\def\fp{\mathfrak{p}}

\def\sZ{\mathsf{Z}}

\def\rS{\mathrm{S}}

\def\Hom{\operatorname{Hom}}
\def\ltensor{\otimes^{\bf{L}}}

\def\Sing{\operatorname{\mathrm{Sing}}}
\def\sup{\operatorname{\mathrm{Supp}}}
\def\ssup{\operatorname{\underline{\mathrm{Supp}}}}
\def\Spcl{\operatorname{\mathbf{Spcl}}}
\def\Spec{\operatorname{\mathrm{Spec}}}
\def\Th{\operatorname{\mathbf{Th}}}
\def\Thom{\operatorname{\mathbf{Thom}}}
\def\Param{\operatorname{\mathbf{Param}}}

\def\tsp{\operatorname{\mathrm{Spec}_{\triangle}}}
\def\tsup{\mathrm{Supp}_{\triangle}}

\def\tRad{\operatorname{\mathbf{Rad}_{\triangle}}}
\def\ttsp{\operatorname{\mathrm{Spec}_{\otimes}}}
\def\ttsup{\mathrm{Supp}_{\otimes}}

\def\kb{\mathrm{K^b}}
\def\dpf{\mathrm{D^{perf}}}
\def\ds{\mathrm{D^{sg}}}

\def\db{\mathrm{D^{b}}}

\def\codim{\operatorname{\mathrm{codim}}}

\def\tpf{\mathcal{S}^{\mathrm{perf}}}
\def\tb{\mathcal{S}^{\mathrm{b}}}
\def\ts{\mathcal{S}^{\mathrm{sg}}}

\def\ol{\overline}

\def\one{\mathbf{1}}

\begin{document}
\title[spectra of triangulated categories]{Prime thick subcategories and spectra of derived and singularity categories of noetherian schemes}
\author{Hiroki Matsui}
\address{Graduate School of Mathematical Sciences\\ University of Tokyo, 3-8-1 Komaba, Meguro-ku, Tokyo 153-8914, Japan}
\email{mhiroki@ms.u-tokyo.ac.jp}
\thanks{The author was partly supported by JSPS Grant-in-Aid for JSPS Fellows 19J00158.}

\begin{abstract}
For an essentially small triangulated category $\cT$, we introduce the notion of prime thick subcategories and define the spectrum of $\cT$, which shares the basic properties with the spectrum of a tensor triangulated category introduced by Balmer. 
We mainly focus on triangulated categories that appear in algebraic geometry such as the derived and the singularity categories of a noetherian scheme $X$.
We prove that certain classes of thick subcategories are prime thick subcategories of these triangulated categories.
Furthermore, we use this result to show that certain subspaces of $X$ are embedded into their spectra as topological spaces.
\end{abstract}

\maketitle

\section{Introduction}
Classification of thick subcategories is one of the important approaches for understanding the structure of a given triangulated category and is a common problem in various fields of mathematics, such as commutative algebra \cite{Hop, Nee, Ste14, Tak10}, algebraic geometry \cite{Ste14, Tho}, modular representation theory of finite groups \cite{BCR, BIK, BIKP}, stable homotopy theory \cite{Hop, HS}, and so on.
Especially, classifications of thick tensor ideals of tensor triangulated categories are quite successful.
For example, the thick tensor ideals of the perfect derived category of a noetherian scheme, those of the stable module category, and those of the stable homotopy category have been classified; see the references above.
Recently, Balmer \cite{Bal05} has established a theory called {\it tensor triangular geometry}, which gives a unified perspective on such classifications.
For a given essentially small tensor triangulated category $(\cT, \otimes, \one)$, Balmer defines a topology on the set $\ttsp(\cT)$ of prime thick tensor ideals of $\cT$, which is called the {\it Balmer spectrum} of $\cT$.
Balmer establishes the following monumental work in the theory.

\begin{thm}[{Balmer}]\label{thm0}
Let $\cT$ be an essentially small tensor triangulated category.
Then there is a lattice isomorphism between the set of radical thick tensor ideals of	$\cT$ and the set of Thomason subsets of $\ttsp(\cT)$.
\end{thm}

\noindent
By virtue of this result, tensor triangular geometry provides us an algebro-geometric way to study an essentially small tensor triangulated category using its Balmer spectrum.

Tensor triangular geometry can be applied to any essentially small tensor triangulated categories, whereas it cannot be directly applied to triangulated categories that are not tensor triangulated.
Such triangulated categories include important ones, e.g., the bounded derived category $\db(X)$ of coherent sheaves and the singularity category $\ds(X)$ of a noetherian scheme $X$ do not have natural tensor triangulated structures, which are used in the study of birational geometry and homological mirror symmetry conjecture; see \cite{Kaw, Orl04} and references therein.
Therefore it is a natural and important problem to develop an analogous theory of tensor triangular geometry for essentially small triangulated categories without tensor triangulated structures.

The present paper aims to construct for a given essentially small triangulated category $\cT$ a topological space $\tsp(\cT)$ which we call the {\it spectrum} of $\cT$.
To do this, we define the notion of {\it prime thick subcategories} of $\cT$ and the spectrum of $\cT$ as the set of prime thick subcategories together with a topology given in \cite{MT1}.
Moreover, we study the prime thick subcategories and spectra of the perfect derived category $\dpf(X)$, the bounded derived category $\db(X)$ of coherent sheaves, and the singularity category $\ds(X)$ of a noetherian scheme $X$.
One of the main theorems of this paper is the following, which means that the spectra contain some topological information on $X$.

\begin{thm}[Corollaries \ref{cor} and \ref{cor2}]\label{main2}
Let $X$ be a noetherian scheme.
\begin{enumerate}[\rm(1)]
\item
There is an immersion
$$
X \hookrightarrow \tsp(\dpf(X))
$$
of topological spaces, which is a homeomorphism if $X$ is quasi-affine.
\item
There is an immersion
$$
\mathrm{CI}(X) \hookrightarrow \tsp(\db(X))
$$
of topological spaces, which is a homeomorphism if $X$ is quasi-affine and regular.
\item
Assume that $X$ is a separated Gorenstein scheme.
Then there is an immersion
$$
\mathrm{HS}(X) \hookrightarrow \tsp(\ds(X))
$$
of topological spaces, which is a homeomorphism if $X$ is quasi-affine and locally a hypersurface.
\end{enumerate}	
\end{thm}
\noindent 
Here, $\mathrm{CI}(X)$ is the {\it complete intersection locus} of $X$, which is the set of points $x$ of $X$ such that $\cO_{X,x}$ are complete intersection.
Similarly, $\mathrm{HS}(X)$ is the {\it hypersurface locus}, which is the set of singular points $x$ such that $\cO_{X,x}$ are hypersurface.
If $X$ is excellent and Cohen-Macaulay, then $\mathrm{CI}(X)$ and $\mathrm{HS}(X)$ are open subsets of
 $X$ and $\Sing(X)$, respectively; see Remark \ref{op}. 

We say that two noetherian schemes $X$ and $Y$ are {\it derived equivalent} if their bounded derived categories $\db(X)$ and $\db(Y)$ are equivalent as triangulated categories.
This concept is also known as {\it Fourier-Mukai partners}.
Since by definition equivalent triangulated categories have homeomorphic spectra, Theorem \ref{main2} immediately implies the following result. 

\begin{cor}[Corollary \ref{cor3}]
Let $X$ be a noetherian	scheme.
Then there is an immersion 
$$
\mathrm{CI}(Y) \hookrightarrow \tsp(\db(X))
$$
of topological spaces for any Fourier-Mukai partner $Y$ of $X$.
\end{cor}

If $\cT$ is a tensor triangulated category, then we have two topological spaces, $\ttsp(\cT)$ and $\tsp(\cT)$.
Although these are not equal in general, there is a relation between prime ideals and prime thick subcategories, which justifies the definition of prime thick subcategories.

\begin{thm}[Proposition \ref{pp} and Corollary \ref{twoprm}]
Let $\cT$ be an essentially small tensor triangulated category and $\cP$ be a radical thick tensor ideal of $\cT$.
If $\cP$ is a prime thick subcategory of $\cT$, then it is a prime thick tensor ideal of $\cT$.
Moreover, the converse holds for $\cT = \dpf(X)$ of a noetherian scheme $X$.
\end{thm}

The organization of this paper is as follows.
In Section 2, we introduce the notion of a prime thick subcategory, give the definition of the spectrum, and study them for a given essentially small triangulated category $\cT$.
Most arguments therein are along the same line as in \cite[Section 2]{MT1}.
In Section 3, we study prime thick subcategories and the spectra of derived and singularity categories $\dpf(X)$, $\db(X)$, and $\ds(X)$.
The proof of Theorem \ref{main2} will be given here.
In Section 4, we investigate our spectrum for an essentially small tensor triangulated category.

\section{Spectra of triangulated categories}

In this section, fix an essentially small triangulated category $\cT$ and let us define the {\it spectrum} $\tsp(\cT)$ of $\cT$. 
A general construction of spectra has been given in \cite{MT1} for a given set of thick subcategories of $\cT$.
We adopt the {\it prime thick subcategories}, which will be defined soon, as the underlying set of the spectrum.
We begin with our convention.

\begin{conv}
\begin{enumerate}[\rm(1)]
\item
Throughout this paper, all triangulated categories $\cT$ are assumed to be essentially small so that the set $\Th(\cT)$ of all thick subcategories forms a set.
Always we consider $\Th(\cT)$ as a lattice via the inclusion relation.
Also, all subcategories are assumed to be full and additive.
Denote by $\mathbf{0}$ the {\it zero subcategory} of $\cT$, that is, the subcategory consisting of objects isomorphic to the zero object.

\item
For a noetherian scheme $X$, the {\it bounded derived category} $\db(X)$ of $X$ is the derived category of bounded complexes of coherent sheaves on $X$.
An object $\cF$ of $\db(X)$ is said to be {\it perfect} if, for any $x \in X$, there is an open neighborhood $U \subseteq X$ of $x$ such that $\cF|_U$ is isomorphic in $\db(U)$ to a bounded complex of free $\cO_U$-modules of finite rank.
Denote by $\dpf(X)$ the thick subcategory of $\db(X)$ consisting of all perfect complexes.
We call it the {\it perfect derived category} of $X$.
Denote by $\ds(X)$ the Verdier quotient $\ds(X) := \db(X)/\dpf(X)$ which we call the {\it singularity category} of $X$.

If we consider an affine scheme $X = \Spec R$, then we simply write $\dpf(R)$, $\db(R)$, and $\ds(R)$ for the perfect derived, the bounded derived, and the singularity categories of $X$.
\end{enumerate}
\end{conv}

Now, we introduce the definition of a prime thick subcategory which plays a central role  throughout the paper.

\begin{dfn}
We say that a thick subcategory $\cP$ of $\cT$ is {\it prime} if there is a unique thick subcategory that is minimal among all thick subcategories $\cX$ of $\cT$ satisfying $\cP \subsetneq \cX$.
Denote by $\tsp(\cT)$ the set of prime thick subcategories of $\cT$.
\end{dfn}

\begin{rem}
\begin{enumerate}[\rm(1)]
\item
We will see in Section 4 that prime thick tensor ideals of a tensor triangulated category are characterized by a similar condition; see Proposition \ref{prid}.
This fact justifies the above definition of prime thick subcategories.
\item
Recently, Takahashi \cite{Tak21} introduced the notion of {\it core} of the singularity category $\ds(R)$ of a commutative noetherian local ring $R$ as the intersection of all non-trivial thick subcategories of $\ds(R)$.
Our definition of a prime thick subcategory and the concept of a core are closely related.
Indeed, one can easily see that $\mathbf{0}$ is a prime thick subcategory of $\ds(R)$ if and only if the core of $\ds(R)$ is not zero.
\end{enumerate}
\end{rem}

We then recall the definition of a topology on $\tsp(\cT)$.

\begin{dfn}[{\cite[Definition 2.1]{MT1}}]
For a family $\cE$ of objects of $\cT$, we set 
$$
\sZ(\cE) := \{\cP \in \tsp (\cT) \mid \cP \cap \cE = \emptyset\}.
$$
\noindent
We can easily check the following conditions hold:
\begin{itemize}
\item
$\sZ(\cT) = \emptyset$ and $\sZ(\emptyset) = \tsp (\cT)$.
\item
$\bigcap_{i \in I} \sZ(\cE_i) = \sZ(\bigcup_{i \in I} \cE_i)$. 
\item	
$\sZ(\cE) \cup \sZ(\cE') = \sZ(\cE \oplus \cE')$, where $\cE \oplus \cE' := \{M \oplus M' \mid M \in \cE, M' \in \cE'\}$. 
\end{itemize}
Here the first and the second conditions hold trivially, and the third one follows since thick subcategories are closed under taking direct sums.
Therefore we can define a topology on $\tsp (\cT)$ whose family of closed subsets are $\{\sZ(\cE) \mid \cE \subseteq \cT\}$.
We call this topological space the {\it spectrum} of $\cT$.

For an object $M \in \cT$, define the {\it support} of $M$ by
$$
\tsup (M) := \sZ(\{M\}) = \{\cP \in \tsp \cT \mid M \not\in \cP \}.
$$
Then the family of supports $\{\tsup (M)\}_{M \in \cT}$ forms a closed basis of $\tsp (\cT)$.
More generally, we set 
$$
\tsup(\cX) := \bigcup_{M \in \cX} \tsup(M)
$$
for a thick subcategory $\cX$ of $\cT$.
\end{dfn}

Here, we list some definitions and general properties which are given in \cite{MT1}.

\begin{prop}[{\cite[Proposition 2.3]{MT1}}] \label{mt1}
For any prime thick subcategory $\cP$ of $\cT$, one has
$$
\overline{\{\cP\}} = \{\cQ \in \tsp(\cT) \mid \cQ \subseteq \cP\}.
$$	
In particular, $\tsp(\cT)$ is a $T_0$-space.
\end{prop}

\begin{dfn}[{\cite[Definition 2.7]{MT1}}] \label{rad}
For a thick subcategory $\cX$ of $\cT$, we define its {\it radical} by
$$
\sqrt{\cX} :=  \bigcap_{\cX \subseteq \cP \in \tsp (\cT)} \cP.
$$	
We say that $\cX$ is {\it radical} if the equality $\sqrt{\cX} = \cX$ holds.
Denote by $\tRad(\cT)$ the set of radical thick subcategories of $\cT$.
\end{dfn}

\begin{dfn}\label{param}
We define the {\it parameter set} of $\cT$ by
$$
\Param(\tsp (\cT)) :=\{\tsup(\cX) \mid \cX \in \Th(\cT)\}.
$$
\end{dfn}

The reason why we call this so is that it parametrizes the radical thick subcategories of $\cT$ as follows.

\begin{thm}[{\cite[Theorem 2.9]{MT1}}]\label{cls}
There is a mutually inverse lattice isomorphisms
$$
\xymatrix{
\tRad(\cT) \ar@<0.5ex>[r]^-{\tsup} &
\Param(\tsp \cT) \ar@<0.5ex>[l]^-{\tsup^{-1}}.
}
$$
Here, we define $\tsup^{-1}(W) := \{M \in \cT \mid \tsup(M) \subseteq W\}$.
\end{thm}

We have explained so far basic definitions and general properties along with \cite{MT1}.
The remainder of this section is devoted to the things that are specific to our definition.

First, we study the functorial properties of spectra.
Let $F: \cT \to \cT'$ be an exact functor of triangulated categories.
Let $\cX \subseteq \cT$ and $\cY \subseteq \cT'$ be thick subcategories.
Then we define the subcategories $F(\cX) \subseteq \cT'$ and $F^{-1}(\cY) \subseteq \cT$ by
\begin{align*}
F(\cX) &:= \{M' \in \cT' \mid M' \cong F(M) \mbox{ for some } M \in \cX\}, \\
F^{-1}(\cY) &:=\{ M \in \cT \mid F(M) \in \cY\}.
\end{align*}
One can easily check that $F^{-1}(\cY)$ is always a thick subcategory of $\cT$.
Therefore, we have a poset homomorphism
$$
\Th(\cT') \to \Th(\cT), \quad \cY \mapsto F^{-1}(\cY).
$$
\begin{prop}\label{quot}
Let $\cT$ be a triangulated category and $\cK$ be its thick subcategory.
Denote by $F: \cT \to \cT/\cK$ the canonical functor.
\begin{enumerate}[\rm(1)]
\item
There is a lattice isomorphism
$$
\Th(\cT/\cK) \xrightarrow{\cong} \{\cX \in \Th(\cT) \mid \cK \subseteq \cX\}, \quad \cQ \mapsto F^{-1}(\cQ).
$$
\item
For a thick subcategory $\cQ \in \Th(\cT/\cK)$, $\cQ$ is a prime thick subcategory of $\cT/\cK$ if and only if $F^{-1}(\cQ)$ is a prime thick subcategory of $\cT$.
\item
$F$ induces an immersion
$$
{}^aF: \tsp(\cT/\cK) \to \tsp(\cT), \quad \cQ \mapsto F^{-1}(\cQ)
$$	
of topological spaces with image $\{\cP \in \tsp(\cT) \mid \cK \subseteq \cP\}$.
\end{enumerate}
\end{prop}

\begin{proof}
(1) The first statement is by \cite[Lemma 3.1]{Tak13}.	

(2) The lattice isomorphism in (1) restricts to a poset isomorphism
$$
\{\cY \in \Th(\cT/\cK) \mid \cQ \subsetneq \cY\} \cong \{\cX \in \Th(\cT) \mid F^{-1}(\cQ) \subsetneq \cX\}.
$$
The left-hand side admits a unique minimal element if and only if so does the right-hand side, that is, $\cQ$ is a prime thick subcategory of $\cT/\cK$ if and only if $F^{-1}(\cQ)$ is a prime thick subcategory of $\cT$.

(3)
Using (1) and (2), we have a well-defined injective map
$$
{}^aF: \tsp(\cT/\cK) \to \tsp(\cT), \quad \cQ \mapsto F^{-1}(\cQ)
$$
with image $\{\cP \in \tsp(\cT) \mid \cK \subseteq \cP\}$.
It is not hard to check that the equality $({}^aF)^{-1}(\tsup(M)) = \tsup(F(M))$ holds for an object $M \in \cT$.
Since the family $\{\tsup(M)\}_{M \in \cT}$ (resp. $\{\tsup(F(M))\}_{M \in \cT}$) forms a closed basis of $\tsp(\cT)$ (resp. $\tsp(\cT/\cK)$), we conclude that ${}^aF: \tsp(\cT/\cK) \to \tsp(\cT)$ is an immersion of topological spaces.
\end{proof}

We say that an exact functor $F: \cT \to \cT'$ of triangulated categories is {\it triangle equivalence up to direct summands} if it is fully faithful and any object $M' \in \cT'$ is a direct summand of $F(M)$ for some $M \in \cT$.
The following lemma is useful.

\begin{lem}[{\cite[(3.2)]{Bal05}}]\label{cof}
Let $F: \cT \to \cT'$ be a triangle equivalence up to direct summands.
Then one has $M' \oplus M'[1] \in F(\cT)$ for any $M' \in \cT$.
\end{lem}

\begin{prop}\label{idm}
Let $F: \cT \to \cT'$ be a triangle equivalence up to direct summands.
\begin{enumerate}[\rm(1)]
\item
There is a lattice isomorphism
$$
\Th(\cT') \xrightarrow{\cong} \Th(\cT), \quad \cY \mapsto F^{-1}(\cY).
$$	
\item
For a thick subcategory $\cQ \in \Th(\cT')$, $\cQ$ is a prime thick subcategory of $\cT'$ if and only if $F^{-1}(\cQ)$ is a prime thick subcategory of $\cT$.
\item
$F$ induces a homeomorphism	
$$
{}^a F : \tsp(\cT') \xrightarrow{\cong} \tsp(\cT), \quad \cP \mapsto F^{-1}(\cP).
$$
\end{enumerate}	
\end{prop}

\begin{proof}
By the same reason as in the proof of Proposition \ref{quot}, it is enough to show that the map	
$$
\Th(\cT') \to \Th(\cT), \quad \cY \mapsto F^{-1}(\cY)
$$
is a lattice isomorphism.	
This follows from the same argument as in \cite[Proposition 3.13]{Bal05}.
\end{proof}

\if0\begin{lem}
Let $F: \cT \to \cT'$ be a triangle equivalence up to direct summands.
Then $F$ induces a lattice isomorphism
$$
\Th(\cT') \xrightarrow{\cong} \Th(\cT), \quad \cY \mapsto F^{-1}(\cY).
$$	
\end{lem}

\begin{prop}\label{idm}
Let $F: \cT \to \cT'$ be a triangle equivalence up to direct summands.
Then there is a homeomorphism
$$
{}^a F : \tsp(\cT') \xrightarrow{\cong} \tsp(\cT), \quad \cP \mapsto F^{-1}(\cP)
$$
\end{prop}

\begin{proof}
It is enough to check that there is a lattice isomorphism
$$
\Th(\cT') \xrightarrow{\cong} \Th(\cT), \quad \cX \mapsto F^{-1}(\cX).
$$
Indeed, for any thick subcategory $\cY \in \Th(\cT')$, this lattice isomorphism induces a poset isomorphism
$$
\{\cX' \in \Th(\cT') \mid \cY \subsetneq \cX'\} \xrightarrow{\cong} \{\cX \in \Th(\cT) \mid F^{-1}(\cY) \subsetneq \cX\}
$$
Then the left-hand side has a unique minimal element if and only if so does the right-hand side.
This means that $\cY$ is a prime thick subcategory of $\cT'$ if and only if so is $F^{-1}(\cY)$.

We follows the argument of \cite[Proposition 3.13]{Bal05}.
For a thick subcategory $\cY \in \Th(\cT)$, write $F(\cY)$ the subcategory of $\cT'$ consisting of object $N \in \cT'$ such that $N \cong F(M)$ for some $M \in \cY$.  
Since $F$ is fully faithful, $F(\cY)$ is a triangulated subcategory of $\cT'$.

First remark that for any object $N \in \cT'$, one has $N \oplus N[1] \in F(\cT)$.

For a thick subcategory $\cY \in \Th(\cT)$, set
$$
\cX:= \{N \in \cT' \mid N \oplus N[1] \in F(\cY)\}.
$$
Then $\cX$ is a thick subcategory of $\cT'$ with $F^{-1}(\cX)= \cY$.
\end{proof}
\fi


Next, we discuss the determination of the spectrum of a given triangulated category using {\it classifying support data}.


\begin{dfn}\label{supd}
Let $\cT$ be a triangulated category.
A {\it support data} for $\cT$ is a pair $(X, \sigma)$ of a topological space $X$ and an assignment $\sigma$ which assigns to an object $M$ of $\cT$ a closed subset $\sigma(M)$ of $X$ satisfying the following conditions;
\begin{enumerate}[\rm(1)]
\item
$\sigma(0) = \emptyset$.
\item
$\sigma(M[n]) = \sigma(M)$ holds for each object $M$ of $\cT$ and integer $n\in \ZZ$.
\item
$\sigma(M) \subseteq \sigma(L) \cup \sigma(N)$ holds for each triangle $L \to M \to N \to L[1]$ in $\cT$.
\item
$\sigma(M \oplus N) = \sigma(M) \cup \sigma(N)$ holds for each pair of objects $M, N$ of $\cT$.
\end{enumerate}
\end{dfn}

A lot of examples of support data for triangulated categories naturally appear in various fields of mathematics.
Here we present the ones that will appear in this paper.

\begin{ex}
\begin{enumerate}[\rm(1)]
\item (\cite[Remark 2.2]{MT1})
The assignment $\cT \ni M \mapsto \tsup(M) \subseteq \tsp(\cT)$ defines a support data $(\tsp(\cT), \tsup)$ for $\cT$.
\item 
Let $X$ be a noetherian scheme.
One can easily verify that the {\it cohomological support} 
$$
\sup_X(\cF) := \{x \in X \mid \cF_x \not\cong 0 \mbox{ in } \db(\cO_{X,x})\}
$$
of $\cF \in \db(X)$ defines a support data $(X, \sup_X)$ for $\db(X)$.
\item 
Let $X$ be a noetherian scheme and denote by $\Sing(X)$ the singular locus of $X$.
Then the {\it singular support} 
\begin{align*}
\ssup_X(\cF) &:= \{x \in \Sing(X) \mid \cF_x \not\cong 0 \mbox{ in } \ds(\cO_{X,x})\} 
\end{align*}
of $\cF \in \ds(X)$ defines a support data $(\Sing(X), \ssup_X)$ for $\ds(X)$, see \cite[Example 2.4]{M}.
\end{enumerate}
\end{ex}

We say that a subset $W$ of a topological space $X$ is {\it specialization-closed} if it is a union of closed subsets of $X$.
We easily see that $W$ is specialization-closed if and only if $\ol{\{x\}} \subseteq W$ for any $x \in W$.
Denote by $\Spcl(X)$ the set of specialization-closed subsets of $X$.

\begin{dfn}\label{csd}
A support data $(X, \sigma)$ for $\cT$ is said to be {\it classifying} if it satisfies the following conditions;
\begin{enumerate}[\rm(i)]
\item
$X$ is a noetherian sober space. 
Here, we say that a topological space $X$ is {\it sober} if every irreducible closed subset has a unique generic point.
\item	
There is a mutually inverse lattice isomorphisms
$$
\xymatrix{
\Th(\cT) \ar@<0.5ex>[r]^-{\sigma} &
\Spcl(X) \ar@<0.5ex>[l]^-{\sigma^{-1}},
}
$$
where $\sigma(\cX) := \bigcup_{M \in \cX} \sigma(M)$ and $\sigma^{-1}(W) := \{M \in \cT \mid \sigma(M) \subseteq W\}$.
\end{enumerate}
\end{dfn}

Using a classifying support data, we can translate a problem concerning 
 thick subcategories to a problem about specialization-closed subsets.
The following lemma characterizes the topological counterpart of prime thick subcategories.

\begin{lem}\label{lem}
Let $X$ be a sober space and $W$ be a specialization-closed subset of $X$.
The following conditions are equivalent.
\begin{enumerate}[\rm(i)]
\item	
There is a unique minimal specialization-closed subset $T$ of $X$ such that $W \subsetneq T$.
\item
There is a unique element $x \in X$ such that $W = \{x' \in X \mid x \not\in \ol{\{x'\}}\}$. 
\end{enumerate}		
\end{lem}

\begin{proof}
(i) $\Rightarrow$ (ii).
Take an element $x \in T \setminus W$.	
Since $T$ is specialization-closed, there are inclusions $W \subseteq W \cup (\overline{\{x\}}\setminus \{x\}) \subsetneq W \cup \overline{\{x\}} \subseteq T$ of specialization-closed subsets of $X$.
From the assumption on $T$, one has the equalities $W = W \cup (\overline{\{x\}}\setminus \{x\})$ and $T = W \cup \overline{\{x\}}$.
These two equalities yield $T = W \cup \overline{\{x\}} = W \cup \{x\}$.
Moreover, such $x$ is uniquely determined by the equality $\{x\} = T \setminus W$.
In summary, there is a unique element $x \in T \setminus W$ which has the following property: each specialization-closed subset $V$ of $X$ with $W\subsetneq V$ satisfies $x \in V$.

Set $W' := \{x' \in X \mid x \not\in \ol{\{x'\}}\}$ and prove the equality $W = W'$.
The inclusion $W \subseteq W'$ is trivial since $W$ is specialization-closed and $x \not\in W$.
On the other hand, one can easily check that $W'$ is a specialization-closed subset of $X$ which does not contain $x$.
Thus we conclude the equality $W=W'$ from the above argument.

(ii) $\Rightarrow$ (i).
First we prove that $T:= W \cup \{x\}$ is specialization-closed, i.e., $\overline{\{x\}} \subseteq W \cup \{x\}$.
For any element $x'\in \overline{\{x\}} \setminus \{x\}$, it satisfies $x \not\in \overline{\{x'\}}$ and hence $x' \in W$.
Here, we used the assumption that $X$ is sober.
Next, we shall check that $T$ is a unique minimal specialization-closed subset that properly contains $W$.
Let $V$ be a specialization-closed subset of $X$ such that $W \subsetneq V$.
Take an element $x' \in V \setminus W$.
Then we have $x \in \overline{\{x'\}} \subseteq V$ and this implies $T = W \cup \{x\} \subseteq V$.
\end{proof} 

Now, we are ready to prove the main theorem in this section.

\begin{thm}\label{rcst}
Let $\cT$ be a triangulated category and $(X, \sigma)$ be a classifying support data for $\cT$.
Then there is a homeomorphism 
$
X \cong \tsp(\cT).
$
\end{thm}

\begin{proof}
By the lattice isomorphism in Definition \ref{csd}(ii), the prime thick subcategories correspond to the specialization-closed subsets $W$ of $X$ which satisfy the equivalent conditions in Lemma \ref{lem}.
Therefore the map
$$
\varphi: X \to \tsp(\cT),\quad x \mapsto \sigma^{-1}(\{x' \in X \mid x \not\in \ol{\{x'\}}\}) 
$$
is a well-defined bijection.
Thus it remains to check that this map is a homeomorphism.
To this end, let us check the equality $\varphi(x) = \{M \in \cT \mid x \not\in \sigma(M)\}$ first.
Let $M$ be an object of $\cT$.
Since $\sigma(M)$ is a closed subset of a noetherian sober space $X$, we can decompose it into the union $\sigma(M) = \overline{\{x_1\}} \cup \cdots \cup \overline{\{x_r\}}$ of irreducible components.
We then obtain the following equivalences, which mean $\varphi(x) = \{M \in \cT \mid x \not\in \sigma(M)\}$:
\begin{align*}
\sigma(M) \subseteq \{x' \in X \mid x \not\in \ol{\{x'\}}\} &\Leftrightarrow x_i \in \{x' \in X \mid x \not\in \ol{\{x'\}}\} \mbox{ for all $i$} \\
&\Leftrightarrow x \not\in \ol{\{x_i\}}\mbox{ for all $i$} \\
&\Leftrightarrow x \not\in \sigma(M).
\end{align*}

For any $x \in X$ and $M \in \cT$, the equality $\varphi(x) = \{M \in \cT \mid x \not\in \sigma(M)\}$ immediately implies 
$
\varphi^{-1}(\tsup(M)) = \sigma(M).
$
Hence $\varphi: X \to \tsp(\cT)$ is continuous.
On the other hand, for any two element $x,x' \in X$, $x \in \ol{\{x'\}}$ if and only if $\varphi(x) \subseteq \varphi(x')$.
Indeed, there are the following equivalences:
\begin{align*}
x \in \ol{\{x'\}} &\Leftrightarrow x' \not\in \sigma(\varphi(x)) (= \{x' \in X \mid x \not\in \ol{\{x'\}}\}) \\
&\Leftrightarrow x' \not\in \sigma(M) \mbox{ for any } M \in \varphi(x) \\
&\Leftrightarrow M \in \varphi(x')\mbox{ for any } M \in \varphi(x) \\
&\Leftrightarrow \varphi(x) \subseteq \varphi(x'). 
\end{align*}
Then Proposition \ref{mt1} yields that $\varphi(\ol{\{x\}}) = \ol{\{\varphi(x)\}}$ and therefore $\varphi^{-1}$ is also continuous.
\end{proof}

Recall that a commutative noetherian local ring $(R, \fm)$ is a {\it hypersurface} if the $\fm$-adic completion of $R$ is a regular local ring modulo a non-zero element. 
Using Theorem \ref{rcst}, we can determine the spectra of $\dpf(X)$ and $\ds(X)$ for some special cases.

\begin{cor}[{cf. \cite[Theorems 3.10 and 4.7]{M}}]\label{cor}
Let $X$ be a noetherian quasi-affine scheme.
\begin{enumerate}[\rm(1)]
\item	
There is a homeomorphism $X \cong \tsp(\dpf(X))$.
\item
Assume further that $X$ is locally a hypersurface (i.e., $\cO_{X,x}$ is a hypersurface for all $x \in X$).
Then there is a homeomorphism $\Sing(X) \cong \tsp(\ds(X))$.
\end{enumerate}
\end{cor}

\begin{proof}
Since $X$ is quasi-affine, the structure sheaf $\cO_{X}$ is an ample line bundle.
According to \cite[Lemma 3.12]{Tho}, every thick subcategory of $\dpf(X)$ containing $\cO_X$ is $\dpf(X)$.

(1)
Let $\cX$ be a thick subcategory of $\dpf(X)$.
Then we can easily verify that the subcategory
$$
\cY:= \{\cF \in \dpf(X) \mid \cF \ltensor_{\cO_X} \cG \in \cX \mbox{ for all } \cG \in \cX\}
$$
is a thick subcategory of $\dpf(X)$ containing $\cO_X$.
Therefore $\cY = \dpf(X)$ and this means that $\cX$ is a thick tensor ideal.
Then the classification \cite[Theorem 3.15]{Tho} shows that $(X, \sup_X)$ is a classifying support data for $\dpf(X)$.
Thus Theorem \ref{rcst} finishes the proof.

(2) 
First note that there is an action $*: \dpf(X) \times \ds(X) \to \ds(X)$ on $\ds(X)$ which makes $\ds(X)$ a $\dpf(X)$-module; see \cite[Section 3]{Ste14}.
Let $\cX$ be a thick subcategory of $\ds(X)$.
Then the subcategory
$$
\cY:= \{\cF \in \dpf(X) \mid \cF * \cG \in \cX \mbox{ for all } \cG \in \cX\}
$$
is a thick subcategory of $\dpf(X)$ containing $\cO_X$.
Again this implies that $\cY = \dpf(X)$ and therefore $\cX$ is a $\dpf(X)$-submodule of $\ds(X)$.
Then the classification \cite[Theorem 7.7]{Ste14} gives us a classifying support data $(\Sing(X), \ssup_X)$ for $\ds(X)$.
The homeomorphism follows by Theorem \ref{rcst}.
\end{proof}

\section{Prime thick subcategories of derived and singularity categories of noetherian schemes}

In this section, we study prime thick subcategories of the derived and the singularity categories of a noetherian scheme $X$ and investigate their spectra.

We start with several lemmas, which allow us to study prime thick subcategories locally.
First, let us discuss a comparison between the Verdier localization and the localization by a prime ideal of a commutative noetherian ring.
Let $R$ be a commutative noetherian ring and $\cT$ be an $R$-linear triangulated category.
We denote by $\cS_\cT(\fp)$ the subcategory of $\cT$ consisting of objects $M$ which satisfy $1_M = 0 \mbox{ in } \Hom_{\cT}(M, M)_\fp$, i.e., $a1_M = 0$ in $\cT$ for some $a  \in R \setminus \fp$.

\begin{lem}\label{loc}
\begin{enumerate}[\rm(a)]
\item
$\cS_{\cT}(\fp)$ is a thick subcategory of $\cT$.
\item
For any element $a \in R \setminus \fp$, the mapping cone $C(a 1_M)$ of $a 1_M: M \to M$ is in $\cS_\cT(\fp)$.
\item
The canonical functor $Q: \cT \to \cT/\cS_\cT(\fp)$ induces an isomorphism
$$
\Hom_{\cT}(M, N)_\fp \xrightarrow{\cong} \Hom_{\cT/\cS_\cT(\fp)}(M, N), \quad f/a \mapsto Q(a1_N)^{-1}Q(f)
$$
of $R$-modules for any $M, N \in \cT$.
\end{enumerate}
\end{lem}

\begin{proof}
(a)
Obviously, $\cS_\cT(\fp)$ is closed under direct summands and shifts.
It remains to show that $\cS_\cT(\fp)$ is closed under extensions.
Take a triangle $L \xrightarrow{f} M \xrightarrow{g} N \to L[1]$ in $\cT$ with $L, N \in \cS_\cT(\fp)$.
Then there are elements $a, b \in R\setminus \fp$ such that  $a 1_L = 0$ and $b 1_N = 0$ in $\cT$.
Since $g (b 1_M) = (b 1_N) g = 0$, there is a morphism $u: M \to L$ such that $f u = b 1_M$.
Therefore $(ab) 1_M = a(fu) = f(a 1_L)u = 0$ and hence $M \in \cS_\cT(\fp)$.

(b)
Set $C:=C(a 1_M)$ and prove $C \in \cS_\cT(\fp)$.
For any object $N \in \cT$, the triangle $M \xrightarrow{a 1_M} M \to C \to M[1]$ induces an exact sequence
\begin{align*}
\Hom_\cT(N, M)_\fp \xrightarrow{\overset{a}{\cong}}  &\Hom_\cT(N, M)_\fp \to \Hom_\cT(N, C)_\fp \\
&\to \Hom_\cT(N, M[1])_\fp \xrightarrow{\overset{a}{\cong}} \Hom_\cT(N, M[1])_\fp
\end{align*}
of $R_\fp$-modules.
Thus we get $\Hom_\cT(N, C)_\fp = 0$ for any $N \in \cT$.
Taking $N =C$, we conclude that $1_{C} = 0 \mbox{ in } \Hom_\cT(C, C)_\fp $.

(c)
Thanks to (b), the canonical functor $Q: \cT \to \cT/\cS_\cT(\fp)$ induces a homomorphism
$$
\Hom_{\cT}(M, N)_\fp \to \Hom_{\cT/\cS_\cT(\fp)}(M, N),\quad f/a \mapsto Q(a1_N)^{-1}Q(f)
$$
of $R$-modules.

Let $f \in \Hom_\cT(M, N)$ with $Q(f) = 0$ in $\cT/\cS_\cT(\fp)$.
Then there is a morphism $s: N \to L$ in $\cT$ such that $sf=0$ in $\cT$ and $C(s) \in \cS_\cT(\fp)$.
Pick an element $a \in R\setminus \fp$ with $a 1_{C(s)} = 0$.
Then $a 1_N$ factors as $a 1_N: N \xrightarrow{s} L \xrightarrow{t} N$.
This shows that $af = (a1_N)f = tsf = 0$ and hence the above homomorphism is injective.

Let $\alpha: M \to N$ be a morphism in $\cT/\cS_\cT(\fp)$.
This morphism is given by $\alpha = Q(s)^{-1}Q(f)$ for some morphisms $f: M \to L$ and $s: N \to L$ in $\cT$ with $C(s) \in \cS_\cT(\fp)$.
Take an element $a \in R \setminus \fp$ such that $a 1_{C(s)}= 0$.
Then $a 1_{N}$ factors as $a 1_{N}:N \xrightarrow{s} L \xrightarrow{t} N$.
Using this factorization, we get $\alpha = Q(s)^{-1}Q(f) = Q(ts)^{-1} Q(tf) = Q(a1_N)^{-1}Q(tf)$, which belongs to the image of the homomorphism.
This shows the surjectivity of the homomorphism.
\end{proof}

Concerning our $R$-linear triangulated categories $\dpf(R)$, $\db(R)$, and $\ds(R)$ for a commutative noetherian ring $R$, the above lemma is translated into the following statement.

\begin{lem}\label{keylem}
Let $R$ be a commutative noetherian ring and $\fp$ be a prime ideal of $R$. 
\begin{enumerate}[\rm(1)]
\item
There is a triangle equivalence
$$
\dpf(R)/\tpf(\fp) \xrightarrow{\cong} \dpf(R_\fp),
$$
where $\tpf(\fp) := \{M \in \dpf(R) \mid M_\fp \cong 0 \mbox{ in } \dpf(R_\fp)\}$.
\item	
There is a triangle equivalence
$$
\db(R)/\tb(\fp) \xrightarrow{\cong} \db(R_\fp),
$$
where $\tb(\fp) := \{M \in \db(R) \mid M_\fp \cong 0 \mbox{ in } \db(R_\fp)\}$.
\item
Assume further that $R$ is a Gorenstein ring of finite Krull dimension.
Then there is a triangle equivalence
$$
\ds(R)/\ts(\fp) \xrightarrow{\cong} \ds(R_\fp),
$$
where $\ts(\fp) := \{M \in \ds(R) \mid M_\fp \cong 0 \mbox{ in } \ds(R_\fp)\}$.
\end{enumerate}
\end{lem}

\begin{proof}
Let $\cD \in \{\dpf, \db, \ds\}$, and $\cS(\fp) \in \{\tpf(\fp), \tb(\fp), \ts(\fp)\}$ the corresponding thick subcategory of $\cD(R)$.	
Then there is an isomorphism 
\begin{align}\tag{$*$}\label{iso}
\Hom_{\cD(R)}(M, N)_\fp \xrightarrow{\cong} \Hom_{\cD(R_\fp)}(M_\fp, N_\fp),\quad f/a \mapsto a^{-1}f_\fp
\end{align}
of $R_\fp$-modules for any $M, N \in \cD(R)$.
Indeed, for $\cD(R) = \dpf(R)$ or $\cD(R) = \db(R)$, see \cite[Lemma 5.2(b)]{AF}.
For $\cD(R) = \ds(R)$, since $R$ is Gorenstein, $\cD(R)$ is equivalent to the stable category of maximal Cohen-Macaulay $R$-modules by \cite[Theorem 4.4.1(2)]{Buc}.
Then (\ref{iso}) follows from \cite[Lemma 3.9]{Yos} for example.
From the isomorphism (\ref{iso}), one has $\cS_{\cD(R)}(\fp) = \cS(\fp)$.
Furthermore, the isomorphism (\ref{iso}) is equal to the composition
$$
\Hom_{\cD(R)}(M, N)_\fp \xrightarrow{\cong} \Hom_{\cD(R)/\cS(\fp)}(M, N) \to \Hom_{\cD(R_\fp)}(M_\fp, N_\fp),
$$
where the first map is the one given in Lemma \ref{loc}(c) and the second one is induced from the localization at $\fp$.
Therefore, the second map is an isomorphism and hence the localization functor $\cD(R) \to \cD(R_\fp)$ induces a fully faithful functor
$$
\cD(R)/\cS(\fp) \to \cD(R_\fp),\quad M \mapsto M_\fp
$$
As a result, it is enough to check that the localization functor $\cD(R) \to \cD(R_\fp)$ is essentially surjective in each case.

(1)
Note that $\dpf(R) \cong \kb(R)$, where $\kb(R)$ is the homotopy category of bounded complexes of finitely generated projective $R$-modules.
Let $P \in \kb(R_\fp)$. 
Then we shall construct $\widetilde{P} \in \kb(R)$ such that $\widetilde{P}_\fp \cong P$ in $\kb(R_\fp)$.
By taking shifts, we may assume that $P := (0 \to P_r \xrightarrow{d_r} P_{r-1} \xrightarrow{d_{r-1}} \cdots \xrightarrow{d_1} P_0 \to 0)$.
Because each component $P_i$ of $P$ is a free $R_\fp$-module, each differential $d_i: P_i \to P_{i-1}$ is given by a matrix with entries in $R_\fp$.
Take $s \in R \setminus \fp$ so that $s\alpha$ belongs to the image of $R$ in $R_\fp$ for all entries $\alpha$ of differentials of $P$.
Then for each $i$, $s d_i$ comes from a homomorphism $\widetilde{d}_i: \widetilde{P}_i \to \widetilde{P}_{i-1}$ between free $R$-modules, i.e., $(\widetilde{P}_i)_\fp = P_i$ and $(\widetilde{d}_i)_\fp = sd_i$.
Since $(\widetilde{d}_{i-1})_\fp (\widetilde{d}_{i})_\fp = s^2d_{i-1} d_{i} =0$, one can find an element $t \in R \setminus \fp$ such that $t\widetilde{d}_i \widetilde{d}_{i-1} =0$ for all $i$.
In particular, we obtain a bounded complex
$$
\widetilde{P} := (0 \to \widetilde{P}_r \xrightarrow{t\widetilde{d}_{r}} \widetilde{P}_{r-1} \xrightarrow{t\widetilde{d}_{r-1}} \cdots \to \widetilde{P}_1  \xrightarrow{t\widetilde{d}_{1}} \widetilde{P}_0 \to 0)
$$
of finitely generated free $R$-modules such that 
$$
\widetilde{P}_\fp = (0 \to P_r \xrightarrow{st d_{r}} P_{r-1} \xrightarrow{st d_{r-1}} \cdots \to P_1 \xrightarrow{st d_{1}} P_0 \to 0).
$$
This $\widetilde{P}_\fp$ is isomorphic to $P$ as complexes of $R_\fp$-modules. 
In fact, there is a commutative diagram
$$
\begin{CD}
0 @>>> P_{r} @>std_{r}>>	 P_{r-1} @>std_{r-1}>> \cdots @>std_{3}>> P_2  @>std_2>> P_1 @>std_1>> P_0 @>>>0 \\
& & @VV(st)^rV @VV(st)^{r-1}V & & @VV(st)^2V @VVstV @| \\ 
0 @>>> P_{r} @>d_{r}>> P_{r-1} @>d_{r-1}>> \cdots @>d_3>> P_2  @>d_2>> P_1 @>d_1>> P_0 @>>> 0 
\end{CD}
$$
whose vertical arrows are all isomorphisms of $R_\fp$-modules.

(2)
Let $M \in \db(R_\fp)$. 
Taking a truncation of a projective resolution of $M$ gives rise to a triangle
$$
N[-1] \xrightarrow{f} P \to M \to N
$$
in $\db(R_\fp)$ with $N$ finitely generated $R_\fp$-module and $P \in \kb(R_\fp)$.
Pick a finitely generated $R$-module $\widetilde{N}$ and $\widetilde{P} \in \kb(R)$ together with isomorphisms $\varphi: \widetilde{N}_\fp \xrightarrow{\cong} N$ and $\psi:\widetilde{P}_\fp \xrightarrow{\cong} P$.
The existence of such an $\widetilde{N}$ is easy and that of a $\widetilde{P}$ is by (1).
As we have remarked as (\ref{iso}), there is an isomorphism
$$
\Hom_{\db(R)}(\widetilde{N}[-1], \widetilde{P})_\fp \xrightarrow{\cong} \Hom_{\db(R_\fp)}(\widetilde{N}_\fp[-1], \widetilde{P}_\fp), \quad g/a\mapsto a^{-1}g_\fp .
$$
Therefore, the composition $\psi^{-1}f\varphi[-1]: \widetilde{N}_\fp[-1] \to \widetilde{P}_\fp$ is equal to $a^{-1}g_\fp$ for some $g \in \Hom_{\db(R)}(\widetilde{N}[-1], \widetilde{P})$ and $a \in R \setminus \fp$.
Embedding $g$ into a triangle
$$
\widetilde{N}[-1] \xrightarrow{g} \widetilde{P} \to \widetilde{M} \to \widetilde{N}
$$
in $\db(R)$, we obtain an object $\widetilde{M} \in \db(R)$ with $\widetilde{M}_\fp \cong M$.
Indeed, localizing this triangle at $\fp$, we obtain a commutative diagram

$$
\xymatrix{
\widetilde{N}_\fp[-1] \ar[r]^{g_\fp} \ar[d]_{a \varphi[-1]}^{\cong} & \widetilde{P}_\fp \ar[r] \ar[d]_{\psi}^{\cong} & \widetilde{M}_\fp \ar[r] \ar@{.>}[d] & \widetilde{N}_\fp \ar[d]_{a \varphi}^{\cong} \\
N[-1] \ar[r]^f & P \ar[r] & M \ar[r] & N 
}
$$
in $\db(R_\fp)$ and hence there is a dotted arrow which is an isomorphism in $\db(R_\fp)$.

(3)
We note that every object of $\ds(R_\fp)$ is isomorphic in $\ds(R_\fp)$ to a finitely generated $R_\fp$-module.
It is isomorphic to a localization of some finitely generated $R$-module. 
\end{proof}

Recall that a commutative noetherian local ring $(R, \fm)$ is a {\it complete intersection (of codimension $c$)} if its $\fm$-adic completion is isomorphic to a regular local ring modulo a regular sequence (of length $c$).
We note that complete intersection local rings of codimension $1$ coincide with hypersurface local rings.

Now, let us first study prime thick subcategories of the perfect derived category $\dpf(X)$ and the bounded derived category $\db(X)$ of a noetherian scheme $X$.
The next proposition gives a characterization of complete intersection local rings in terms of prime thick subcategories of the bounded derived category.

\begin{prop}\label{lem1}
Let $(R,\fm)$ be a commutative noetherian local ring.
Then $\mathbf{0}$ is a prime thick subcategory of $\db(R)$ if and only if $R$ is a complete intersection.
\end{prop}

\begin{proof}
Assume that $\mathbf{0}$ is a prime thick subcategory of $\db(R)$ and let $\cX$ be a unique minimal non-zero thick subcategory of $\db(R)$.
The minimality of $\cX$ shows that $\cX \subseteq \dpf(R)$ and in particular $\cX$ contains a non-zero perfect complex.
For any non-zero object $M \in \db(R)$, denote by $\mathrm{thick} (M)$ the smallest thick subcategory of $\db(R)$ containing $M$.
Since $M$ is non-zero, again using the minimality of $\cX$, we get $\cX \subseteq \mathrm{thick} (M)$.
This shows that $\mathrm{thick} (M)$ contains a non-zero perfect complex for any non-zero object $M \in \db(R)$.
Consequently, the `only if' part follows by the result \cite[Theorem 5.2]{Pol}.

To show the `if part', assume that $R$ is a complete intersection.
The combination of \cite[Theorem 1.5]{Nee} and \cite[Theorem 5.2]{Pol} yields that every non-zero thick subcategory $\cX$ of $\db(R)$ contains the non-zero thick subcategory $\sup_{\dpf(R)}^{-1}(\{\fm\}):= \{M \in \dpf(R) \mid \sup_R(M) \subseteq \{\fm\}\}$.
Hence, $\sup_{\dpf(R)}^{-1}(\{\fm\})$ is a unique minimal non-zero thick subcategory of $\db(R)$. 
\end{proof}

Now we are ready to prove one of the main theorems in this paper, which provides examples of prime thick subcategories.

\begin{thm}\label{thder}
Let $X$ be a noetherian scheme and $x \in X$.
\begin{enumerate}[\rm(1)]
\item
The thick subcategory
$$
\tpf_X(x) := \{\cF \in \dpf(X) \mid \cF_x \cong 0 \mbox{ in } \dpf(\cO_{X,x})\}
$$
is a prime thick subcategory of $\dpf(X)$.
\item
Set
$$
\tb_X(x) := \{\cF \in \db(X) \mid \cF_x \cong 0 \mbox{ in } \db(\cO_{X,x})\}.
$$
Then $\tb_X(x)$ is a prime thick subcategory of $\db(X)$ if and only if $\cO_{X,x}$ is a complete intersection.
\end{enumerate}	
\end{thm}

\begin{proof}
(1)
Let $U \subseteq X$ be an affine open neighborhood of $x$ with complement $Z := X \setminus U$. 
It follows from \cite[Theorem 2.13]{Bal02} that the restriction functor $(-)|_U: \dpf(X) \to \dpf(U)$ induces a triangle equivalence
$$
\dpf(X)/\mathrm{D_{\rm Z}^{perf}}(X) \to \dpf(U)
$$
up to direct summands, where $\mathrm{D_{\it Z}^{perf}}(X) := \{\cF \in \dpf(X) \mid \cF|_U \cong 0 \mbox{ in } \dpf(U)\}$.
On the other hand, $\mathrm{D_{\it Z}^{perf}}(X) \subseteq \tpf_X(x)$ follows from their definitions.
By Proposition \ref{idm}(1),  we have the following lattice isomorphism
$$
\Th(\dpf(X)/\mathrm{D_{\rm Z}^{perf}}(X)) \cong \Th(\dpf(U))
$$
under which $\tpf_X(x)/\mathrm{D_{\rm Z}^{perf}}(X)$ corresponds to $\tpf_U(x)$.
As a result, we may assume that $X$ is affine by Proposition \ref{quot}(2).

Let $X$ be an affine noetherian scheme. 
From Proposition \ref{quot}(2) and the triangle equivalence $\dpf(X)/\tpf(x) \cong \dpf(\cO_{X,x})$ which is obtained in Lemma \ref{keylem}(1), it is enough to check that $\mathbf{0}$ is a prime thick subcategory of $\dpf(\cO_{X,x})$.
We use the poset isomorphism
$$
\{\cX \in \Th(\dpf(\cO_{X,x})) \mid \cX \neq \mathbf{0}\} \cong \{W \in \Spcl(\Spec \cO_{X,x}) \mid W \neq \emptyset\}
$$ 
which follows from \cite[Theorem 1.5]{Nee}.
Since the right-hand side has a unique minimal element $\{\fm_{X,x}\}$ ($\fm_{X,x}$ is the unique maximal ideal of $\cO_{X,x}$), so does the left-hand side.
This means that $\mathbf{0}$ is a prime thick subcategory of $\dpf(\cO_{X,x})$.

(2)
Let $U \subseteq X$ be an affine open neighborhood of $x$ with complement $Z := X \setminus U$. 
By \cite[Theorem 3.3.2]{Sch}, the restriction functor $(-)|_U: \db(X) \to \db(U)$ induces a triangle equivalence 
$$
\db(X)/\mathrm{D_{\it Z}^b}(X) \xrightarrow{\cong} \db(U),
$$
where $\mathrm{D_{\it Z}^b}(X) := \{\cF \in \db(X) \mid \cF|_U = 0 \mbox{ in } \db(U)\}$.
With the same argument as above, this triangle equivalence and Proposition \ref{quot}(2) allow us to assume that $X$ is affine.
From this observation, the result follows by Propositions \ref{quot}(2), \ref{lem1} and Lemma \ref{keylem}(2).
\end{proof}

Next, we discuss the singularity category of a noetherian scheme and prove a similar result to Theorem \ref{thder}.
To reduce to the affine case, we need the following result. 
Although it may be well known, we give a proof since we have not been able to locate a specific reference.

\begin{prop}\label{resds}
Let $X$ be a separated noetherian scheme and $U$ an affine open subset of $X$ with complement $Z:=X \setminus U$.
Then the restriction functor $(-)|_U: \ds(X) \to \ds(U)$ induces a triangle equivalence
$$
\ds(X)/\mathrm{D_{\mathit{Z}}^{sg}}(X) \xrightarrow{\cong} \ds(U),
$$	
where $\mathrm{D_{\mathit{Z}}^{sg}}(X) := \{\cF \in \ds(X)\mid \cF|_U \cong 0 \mbox{ in $\ds(U)$}\}$.
\end{prop}
 
This result is essentially shown by Krause in \cite[Proposition 6.9]{Kra} using the {\it unbounded stable derived category} $\rS(X)$.
Let $X$ be a separated noetherian scheme.
Denote by $\mathrm{K}(\operatorname{\mathrm{Inj}}X)$ the homotopy category of complexes of injective quasi-coherent $\cO_X$-modules, and by $\mathrm{S}(X)$ the subcategory of $\mathrm{K}(\operatorname{\mathrm{Inj}}X)$ consisting of all acyclic complexes.
By the result of \cite[Corollary 5.4]{Kra}, $\mathrm{S}(X)$ is compactly generated triangulated category and there is a triangle equivalence
$$
F:\ds(X) \to \rS(X)^c
$$
up to direct summands, where $(-)^c$ stands for the thick subcategory consisting of all compact objects.

\begin{proof}[Proof of Proposition \ref{resds}]
Let $j:U \hookrightarrow X$ be the open immersion.
Since $j^*: \mathrm{K}(\operatorname{\mathrm{Inj}}X) \to \mathrm{K}(\operatorname{\mathrm{Inj}}U)$ preserves acyclicity, it restricts to an exact functor $j^*: \rS(X) \to \rS(U)$.
By virtue of \cite[Corollary 4.11]{Ste13} and \cite[Lemma 7.2]{Ste14}, $\rS_Z(X)$ is generated by compact objects of $\rS(X)$.
Then it follows from \cite[Corollary 6.9]{Kra} and \cite[Theorem 2.1]{NTTY} that $j^*: \rS(X) \to \rS(U)$ induces a triangle equivalence
$$
\ol{j^*}:\rS(X)^c/\rS_Z(X)^c \to \rS(U)^c
$$
up to direct summands.
Note that by \cite[Theorem 6.6]{Kra}, there is a commutative diagram
\begin{align}\tag{$*$}\label{diag}
\begin{aligned}
\xymatrix{
\mathrm{D_{\mathit{Z}}^{sg}}(X) \ar@{^{(}-_>}[r] \ar[d]_F & \ds(X) \ar[r]^-{j^*} \ar[d]_F & \ds(U) \ar[d]_F \\
\rS_Z(X)^c \ar@{^{(}-_>}[r] & \rS(X)^c \ar[r]^-{j^*} & \rS(U)^c
}
\end{aligned}
\end{align}
of exact functors and hence we get $F^{-1}(\rS_Z(X)^c) = \mathrm{D_{\mathit{Z}}^{sg}}(X)$.

Consider a morphism $f:M' \to F(M)$ in $\rS(X)^c$ with $M' \in \rS_Z(X)^c$ and $M \in \ds(X)$.
By Lemma \ref{cof}, there is an isomorphism $\varphi: M' \oplus M'[1] \xrightarrow{\cong} F(L)$ for some $L\in \ds(X)$.
Since $F(j^*(L))\cong j^*(F(L)) \cong j^*(M' \oplus M'[1]) = 0$, one has $j^*(L) \cong 0$ i.e., $L \in \mathrm{D_{\mathit{Z}}^{sg}}(X)$.
Therefore, $f$ factors as $M' \xrightarrow{\varphi {}^t(1,0)} F(L) \xrightarrow{(f,0)\varphi^{-1}} F(M)$ with $L \in \mathrm{D_{\mathit{Z}}^{sg}}(X)$.
Using \cite[Proposition 10.2.6(ii)]{KaSc}, $F$ induces a triangle equivalence 
$$
\ol{F}: \ds(X)/\mathrm{D_{\mathit{Z}}^{sg}}(X) \to \rS(X)^c/\rS_Z(X)^c 
$$
up to direct summands.
From the commutative diagram (\ref{diag}), we have a commutative diagram
$$
\xymatrix{
\ds(X)/\mathrm{D_{\mathit{Z}}^{sg}}(X) \ar[r]^-{\ol{j^*}} \ar[d]_{\ol{F}} & \ds(U) \ar[d]_F \\
\rS(X)^c/\rS_Z(X)^c \ar[r]^-{\ol{j^*}} & \rS(U)^c
}
$$
of exact functors, where the vertical functors and the bottom $\ol{j^*}$ are triangle equivalences up to direct summands.
Therefore, the top $\ol{j^*}$ is also fully faithful.
On the other hand, $\ol{j^*}: \ds(X)/\mathrm{D_{\mathit{Z}}^{sg}}(X) \to \ds(U)$ is essentially surjective as the restriction functor $j^*: \db(X) \to \db(U)$ is essentially surjective by \cite[Theorem 3.3.2]{Sch}.
Hence we conclude that $\ol{j^*}: \ds(X)/\mathrm{D_{\mathit{Z}}^{sg}}(X) \to \ds(U)$ is a triangle equivalence.
\end{proof}

The next proposition, which is similar to Proposition \ref{lem1}, characterizes hypersurface local rings in terms of prime thick subcategories of the singularity category of a commutative noetherian local ring.

\begin{prop}\label{lem2}
Let $R$ be a non-regular commutative noetherian local ring.
If $R$ is a hypersurface, then $\mathbf{0}$ is a prime thick subcategory of $\ds(R)$.
The converse holds if $R$ is a complete intersection.	
\end{prop}

\begin{proof}
Assume that $R$ is hypersurface.
From \cite[Theorem 6.6]{Tak10}, there is a poset isomorphism
$$
\{\cX \in \Th(\ds(R)) \mid \cX \neq \mathbf{0}\} \cong \{W \in \Spcl(\Sing(R)) \mid W \neq \emptyset\},
$$
which proves the first statement.
The second one follows from 	\cite[Theorem 3.7]{Tak21}.
Indeed, if $\mathbf{0}$ is prime, then there is a unique minimal non-zero thick subcategory $\cX$ of $\ds(R)$.
This $\cX$ must be equal to the intersection of all non-zero thick subcategories of $\ds(R)$.
\end{proof}

In view of the proof of Theorem \ref{thder}, we can deduce the following result from Lemma \ref{keylem}(3) and Propositions \ref{resds}, \ref{lem2}.

\begin{thm}\label{thds}
Let $X$ be a separated Gorenstein scheme and $x \in \Sing(X)$.
If $\cO_{X,x}$ is a hypersurface, then 	
$$
\ts_X(x) := \{\cF \in \ds(X) \mid \cF_x \cong 0 \mbox{ in } \ds(\cO_{X,x})\}
$$
is a prime thick subcategory of $\ds(X)$.
Conversely, if $\cO_{X,x}$ is a complete intersection and $\ts_X(x)$ is a prime thick subcategory of $\ds(X)$, then $\cO_{X,x}$ is a hypersurface.
\end{thm}

\begin{proof}
Using Proposition \ref{resds}, the same argument as in the proof of Theorem \ref{thder} allows us to reduce to the affine case.
Assume $X$ is an affine scheme.
Then Lemma \ref{keylem}(3) gives us a triangle equivalence
$$
\ds(X)/\ts_X(x) \cong \ds(\cO_{X,x}).
$$
It follows from Proposition \ref{quot}(2) that $\ts_X(x)$ is prime if and only if $\mathbf{0}$ is a prime thick subcategory of $\ds(\cO_{X,x})$.
Then the result follows by Proposition \ref{lem2}.
\end{proof}

For a noetherian scheme $X$, we define the {\it complete intersection locus} $\mathrm{CI}(X)$ and the {\it hypersurface locus} $\mathrm{HS}(X)$ of $X$ by 
\begin{align*}
\mathrm{CI}(X) &:= \{x \in X \mid \cO_{X,x} \mbox{ is a complete intersection} \}, \\
\mathrm{HS}(X) &:= \{x \in \Sing(X) \mid \cO_{X,x} \mbox{ is a hypersurface} \}.
\end{align*}

Applying Theorems \ref{thder} and \ref{thds}, we obtain the following result which has been stated in the introduction.

\begin{cor}\label{cor2}
Let $X$ be a noetherian scheme.
\begin{enumerate}[\rm(1)]
\item
There is an immersion
$$
X \hookrightarrow \tsp(\dpf(X))
$$
of topological spaces.
\item
There is an immersion
$$
\mathrm{CI}(X) \hookrightarrow \tsp(\db(X))
$$
of topological spaces.
\item
If $X$ is a separated Gorenstein scheme, then there is an immersion
$$
\mathrm{HS}(X) \hookrightarrow \tsp(\ds(X))
$$
of topological spaces.
\end{enumerate}	
\end{cor}

\begin{proof}
(1)
From the equality $\sup_X(\tpf_X(x)) = \{x' \in X \mid x \not\in \ol{\{x'\}}\}$ and Theorem \ref{thder}(1), there is an injective map
$$
X \hookrightarrow \tsp(\dpf(X)),\quad x \mapsto \tpf_X(x).
$$
For an object $\cF \in \dpf(X)$, one can easily see that the inverse image of $\tsup(\cF)$ by this injection is $\sup_X(\cF)$.
Since the family $\{\sup_X(\cF) \mid \cF \in \dpf(X)\}$ (resp. $\{\tsup(\cF) \mid \cF \in \dpf(X)\}$) forms a closed basis of $X$ (resp. $\tsp(\dpf(X))$), the topology on $X$ is nothing but the one induced from $\tsp(\dpf(X))$.

The remained statements (2) and (3) are shown similarly as above.
\end{proof}

\begin{rem}\label{hsop}\label{op}
If $X$ is an excellent Cohen-Macaulay scheme, then the complete intersection locus $\mathrm{CI}(X)$ and the hypersurface locus $\mathrm{HS}(X)$ of $X$ are open subsets of $X$ and $\Sing(X)$, respectively. 

Indeed, because the problem is local, we may assume that $R$ is an excellent Cohen-Macaulay ring and prove that the complete intersection locus $\mathrm{CI}(R)$ and the hypersurface locus $\mathrm{HS}(R)$ are open in $\Spec R$ and $\Sing R$, respectively.
The openness of the complete intersection locus $\mathrm{CI}(R)$ has been already known in \cite[Corollary 3.3]{GM}.
Therefore, assume that $R$ is an excellent locally complete intersection ring and show that the hypersurface locus $\mathrm{HS}(R)$ is open in $\Sing R$.

Set 
$
\mathrm{H}_1(R) := \{\fp \in \Spec R \mid \codim R_\fp \le 1\}.
$
Since $\mathrm{HS}(R) = \mathrm{H}_1(R) \cap \Sing R$, it suffices to show that $\mathrm{H}_1(R)$ is open in $\Spec R$.
Recall that for a prime ideal $\fp$ of $R$, $\codim R_\fp$ coincides with the {$1$st deviation} $\epsilon_1(R_\fp)$ because $R_\fp$ is a complete intersection; see \cite[Section 21]{Mat}.
Then the openness of $\mathrm{H}_1(R)$ follows from \cite[Proposition 3.6]{Rag}.
\end{rem}

We say that two noetherian schemes $X$ and $Y$ are
\begin{itemize}
\item
{\it perfectly derived equivalent} if $\dpf(X) \cong \dpf(Y)$ as triangulated categories.
\item
{\it derived equivalent} if $\db(X) \cong \db(Y)$ as triangulated categories.

\item	
{\it singularly equivalent} if $\ds(X) \cong \ds(Y)$ as triangulated categories.
\end{itemize}
If $X$ and $Y$ are derived equivalent, then $Y$ is said to be a {\it Fourier-Mukai partner} of $X$.
The following result is a direct consequence of Corollary \ref{cor2} because equivalent triangulated categories have homeomorphic spectra.

\begin{cor}\label{cor3}
Let $X$ be a noetherian scheme.
\begin{enumerate}[\rm(1)]
\item
There is an immersion
$$
Y \hookrightarrow \tsp(\dpf(X))
$$
of topological spaces for any noetherian scheme $Y$ which is perfectly derived equivalent to $X$.
\item
There is an immersion
$$
\mathrm{CI}(Y) \hookrightarrow \tsp(\db(X))
$$
of topological spaces for any noetherian scheme $Y$ which is derived equivalent to $X$.
\item	
Let $X$ be a separated Gorenstein scheme.
There is an immersion 
$$
\mathrm{HS}(Y) \hookrightarrow \tsp(\ds(X))
$$
of topological spaces for any separated Gorenstein scheme $Y$ which is singularly equivalent to $X$.
\end{enumerate}
\end{cor}

\begin{rem}
For two commutative noetherian rings $R$ and $S$, it is known as the derived Morita theorem that there are implications
$$
\mbox{perfectly derived equivalent} \Leftrightarrow \mbox{derived equivalent} \Rightarrow \mbox{singularly equivalent}.
$$

If $X$ and $Y$ are projective schemes over a field and assume that $X$ possesses an ample line bundle satisfying a certain vanishing condition on cohomologies, then the same implications
$$
\mbox{perfectly derived equivalent} \Leftrightarrow \mbox{derived equivalent} \Rightarrow \mbox{singularly equivalent}
$$
hold by \cite[Theorem 7.13]{Ba}.
\end{rem}

\section{Comparison with Balmer spectra}

In this section, we compare the spectrum $\tsp(\cT)$ discussed so far to the Balmer spectrum $\ttsp(\cT)$ for a tensor triangulated category $(\cT, \otimes, \one)$. 
Besides, we discuss more the spectrum of the tensor triangulated category $(\dpf(X), \ltensor_{\cO_X}, \cO_X)$ for a noetherian scheme $X$.
To this end, let us start with a brief survey on Balmer's tensor triangular geometry.

Recall that a triple $(\cT, \otimes, \one)$ is a {\it tensor triangulated category} if $\cT$ is a triangulated category equipped with a symmetric monoidal structure $(\otimes, \one)$ which is compatible with the triangulated structure; see \cite[Appendix A]{HPS} for the precise definition.

\begin{dfn}
Let $(\cT, \otimes, \one)$ be a tensor triangulated category.
\begin{enumerate}[\rm(1)]
\item
A thick subcategory $\cI$ of $\cT$ is called a {\it (thick tensor) ideal} if for $M \in \cT$ and $N \in \cX$ one has $M\otimes N \in \cX$.
\item
For an ideal $\cI$ of $\cT$, define its {\it radical} by 
$$
\sqrt{\cI} :=\{M \in \cT \mid M^{\otimes n} \in \cT \mbox{ for some non-negative integer $n$} \}.
$$
We say that $\cI$ is {\it radical} if $\sqrt{\cI} = \cI$ holds.
The set of radical ideals of $\cT$ is denoted by $\mathbf{Rad}_{\otimes}(\cT)$.
\item
An ideal $\cP$ of $\cT$ is called {\it prime} provided $\cP \neq \cT$ and if $M \otimes N$ is in $\cP$, then so is either $M$ or $N$.	
Denote by $\ttsp (\cT)$ the set of prime ideals of $\cT$.
\end{enumerate}
\end{dfn}

Now, let us define a topology on $\ttsp(\cT)$ whose definition is the same as those of $\tsp(\cT)$, except it uses prime ideals instead of prime thick subcategories.
 
\begin{dfn}[{\cite[Definition 2.1]{Bal05}}]
For a family $\cE$ of objects of $\cT$, we set 
$$
\sZ(\cE) := \{\cP \in \ttsp(\cT) \mid \cP \cap \cE = \emptyset\}.
$$
We can easily check that the family $\{\sZ(\cE) \mid \cE \subseteq \cT\}$ of subsets of $\ttsp(\cT)$ satisfies the axioms for closed subsets.
The set $\ttsp (\cT)$ together with this topology is called the {\it Balmer spectrum} of $\cT$.

For an object $M \in \cT$, define the {\it Balmer support} of $M$ by
$$
\ttsup(M) := \sZ(\{M\}) = \{\cP \in \ttsp(\cT) \mid M \not\in \cP\}.
$$
Then the family of the Balmer supports forms a closed basis of $\ttsp(\cT)$. 
\end{dfn}

Proposition \ref{mt1}, Definition \ref{rad}, and Theorem \ref{cls} should be compared with the following couple of results.

\begin{prop}[{\cite[Proposition 2.9]{Bal05}}]\label{cl}
For a prime ideal $\cP$ of $\cT$, one has
$$
\overline{\{\cP\}} = \{\cQ \in \ttsp(\cT) \mid \cQ \subseteq \cP\}.
$$	
In particular, $\tsp(\cT)$ is a $T_0$-space.
\end{prop}

\begin{prop}[{\cite[Lemma 4.2]{Bal05}}]\label{int}
Let $\cI$ be an ideal of $\cT$.
Then one has the equality
$$
\sqrt{\cI} =\bigcap_{\cI \subseteq \cP \in \ttsp(\cT)} \cP.
$$	
\end{prop}
 
A subset of a topological space $X$ is said to be {\it Thomason} if it is a union of closed subsets whose complements are quasi-compact.
We denote by $\Thom(X)$ the set of Thomason subsets of $X$.
It is immediately from the definition that Thomason subsets are specialization-closed.
If $X$ is noetherian, then the converse also holds true i.e., $\Thom(X) = \Spcl(X)$.

\begin{thm}[{\cite[Theorem 4.10]{Bal05}}]\label{bal}
Let $\cT$ be a tensor triangulated category.
Then there is a mutually inverse lattice isomorphisms
$$
\xymatrix{
\mathbf{Rad}_{\otimes}(\cT) \ar@<0.5ex>[r]^-{\ttsup} &
\Thom(\ttsp (\cT)) \ar@<0.5ex>[l]^-{\ttsup^{-1}},
}
$$
where $\ttsup(\cX) := \bigcup_{M \in \cX} \ttsp(M)$ and $\ttsup^{-1}(W) := \{M \in \cT \mid \ttsup(M) \subseteq W\}$.
\end{thm}

\begin{rem}[cf. Definition \ref{param}]
We remark that by \cite[Proposition 2.14]{Bal05} the equality
$$
\Thom(\ttsp (\cT)) =\{\ttsup(\cX) \mid \subseteq \cX\}.
$$
holds.
\end{rem}

The following result shows that prime ideals of $\cT$ are characterized by a similar condition to the definition of prime thick subcategories.  
Therefore, it justifies our definition of a prime thick subcategory. 
 
\begin{prop}\label{prid}
Let $\cP$ be a radical ideal of $\cT$.
Assume that there is a unique minimal radical thick subcategory $\cI$ of $\cT$ with $\cI \subsetneq \cT$.
Then $\cP$ is a prime ideal.
The converse holds if $\ttsp(\cT)$ is noetherian (e.g., $\cT = \dpf(X)$ for a noetherian scheme).
\end{prop}

\begin{proof}
Since $\cP$ is a radical ideal, by Proposition \ref{int}, $\cP$ is the intersection of all prime ideals $\cQ$ which contain $\cP$.
If $\cP$ is not a prime ideal, these $\cQ$ properly contain $\cP$ and hence also contain $\cI$ by assumption.
Therefore, we conclude that $\cP \subsetneq \cI \subseteq \sqrt{\cP} = \cP$, a contradiction.
As a result, $\cP$ is a prime ideal.

Assume that $\ttsp(\cT)$ is noetherian and $\cP$ is a prime ideal.
By the lattice isomorphism in Theorem \ref{bal}, $\cP$ corresponds to the specialization-closed subset 
$$
W := \{\cQ \in \ttsp(\cT) \mid \cP \not\subseteq \cQ\} = \{\cQ \in \ttsp(\cT) \mid \cP \not\in \ol{\{\cQ\}}\}.
$$
Here, the second equality follows from Proposition \ref{cl}.
Then Lemma \ref{lem} shows that there is a unique minimal specialization-closed subset $T$ of $\ttsp(\cT)$ with $W \subsetneq T$.
Again using the lattice isomorphism in Theorem \ref{bal}, $\cI:= \ttsup^{-1}(T)$ satisfies the condition in the statement.
\end{proof}

Now we establish the following result which gives a connection between prime ideals and prime thick subcategories.

\begin{prop}\label{pp}
Let $\cT$ be a tensor triangulated category and $\cP$ be a radical ideal of $\cT$.
If $\cP$ is a prime thick subcategory of $\cT$, then it is a prime ideal of $\cT$.
\end{prop}

\begin{proof}
As $\cP$ is a prime thick	 subcategory of $\cT$, there is a unique minimal thick subcategory $\cX$ of $\cT$ such that $\cX \subsetneq \cT$.
Denote by $\widetilde{\cX}$ the smallest radical ideal of $\cT$ containing $\cX$.
For any radical ideal $\cI$ of $\cT$ with $\cP \subsetneq \cI$, the minimality of $\cX$ shows $\cX \subseteq \cI$ and hence $\widetilde{\cX} \subseteq \cI$.
This means that $\widetilde{\cX}$ is a unique minimal radical ideal of $\cT$ with $\cP \subsetneq \widetilde{\cX}$.
Then Proposition \ref{prid} shows that $\cP$ is a prime ideal.
\end{proof}

For a noetherian scheme $X$, the thick subcategory $\tpf_X(x)$ is a prime ideal, and actually, every prime ideal of $\dpf(X)$ is given in this way; see \cite[Corollary 5.6]{Bal05}.
Therefore, the immersion in Theorem \ref{thder}(1) can be considered as the inclusion
$$
\ttsp(\dpf(X)) \subseteq \tsp(\dpf(X)).
$$
From this observation and Proposition \ref{pp}, we obtain the following corollary.

\begin{cor}\label{twoprm}
Let $X$ be a noetherian scheme and $\cP$ be an ideal of $\dpf(X)$.
Then $\cP$ is a prime thick subcategory of $\dpf(X)$ if and only if $\cP$ is a prime ideal of $\dpf(X)$.	
\end{cor}

We close this paper by giving one concrete example of spectra.
\begin{ex}
Let $k$ be a field.
First note that since $\PP^1$ is regular, there is a triangle equivalence $\db(\PP^1) \cong \dpf(\PP^1)$.
It is explained in \cite[Section 4.1]{KS} that there is a lattice isomorphism
$$
\Th(\dpf(\PP^1)) \cong \Spcl(\PP^1) \sqcup \ZZ.
$$	
Here, $\ZZ$ is considered as the discrete lattice.
On the right-hand side, an element of $\Spcl(\PP^1)$ corresponds to an ideal of $\dpf(\PP^1)$ and those of $\ZZ$ corresponds to a thick subcategory of the form $\mathrm{thick}(\cO_{\PP^1}(i))$ for some $i \in \ZZ$.
As Corollary \ref{twoprm} shows, the prime ideals and the ideals which are prime thick subcategories are the same.
Therefore, restricting the above lattice isomorphism to prime thick subcategories, we get a homeomorphism 
$$
\tsp(\dpf(\PP^1)) \cong \PP^1 \sqcup \ZZ,
$$	
where $\ZZ$ is considered as the discrete topological space.
This gives an example of noetherian scheme which is not quasi-affine and the immersion $X \hookrightarrow \tsp(\dpf(X))$ is not a homeomorphism.
\end{ex}

\begin{ac}
The author thanks Ryo Takahashi for giving helpful comments and useful suggestions. 
The author also thanks Tsutomu Nakamura for fruitful discussions on Proposition \ref{resds}.
\end{ac}

\bibliography{ref}
\bibliographystyle{amsplain}

\end{document}